\documentclass[11pt, reqno]{amsart}

\usepackage{amsthm,amssymb,amstext,amscd,amsfonts,amsbsy,amsrefs,amsxtra,latexsym,amsmath,xcolor,mathrsfs,fancybox,upgreek, soul,url}
\usepackage[english]{babel}
\usepackage[all,cmtip]{xy}
\usepackage[latin1]{inputenc}
\usepackage{cancel}
\usepackage[draft]{hyperref}

\usepackage{comment}
\usepackage{mdframed}
\allowdisplaybreaks
\usepackage{mathtools}
\usepackage{enumerate}
\usepackage{thmtools}
\usepackage{thm-restate}
\usepackage{chngcntr}
\usepackage{enumitem}
\usepackage{etoolbox}
\usepackage{todonotes}

\usepackage{tikz}
\usepackage{verbatim}
\usetikzlibrary{quotes,angles}
\usepackage{pgfplots}
\pgfplotsset{compat=1.12}
\usepgfplotslibrary{fillbetween}
\usetikzlibrary{intersections}
\usetikzlibrary{patterns}
\usepackage{circuitikz}

\DeclarePairedDelimiter\abs{\lvert}{\rvert}
\DeclarePairedDelimiter\norm{\lVert}{\rVert}

\makeatletter
\let\oldabs\abs
\def\abs{\@ifstar{\oldabs}{\oldabs*}}
\let\oldnorm\norm
\def\norm{\@ifstar{\oldnorm}{\oldnorm*}}
\makeatother

\makeatletter
\glb@settings 
\makeatother

\newtheorem{theorem}{Theorem}
\newtheorem{lemma}[theorem]{Lemma}

\newtheorem{proposition}[theorem]{Proposition}

\theoremstyle{definition}

\theoremstyle{remark}

\newtheorem*{remark}{Remark}
\newtheorem*{remarks}{Remarks}

\numberwithin{theorem}{section}
\numberwithin{proposition}{section}
\numberwithin{lemma}{section}
\numberwithin{corollary}{section}
\numberwithin{equation}{section}
\numberwithin{conjecture}{section}

\setlist[enumerate,1]{before=}
\AfterEndEnvironment{enumerate}{}

\newcommand{\N}{\mathbb{N}}
\newcommand{\Z}{\mathbb{Z}}

\newcommand{\C}{\mathbb{C}}

\usepackage{bm}

\def\H{\mathbb{H}}

\begin{document}

\title[Bivariate asymptotics for eta-theta quotients with simple poles]{Bivariate asymptotics for eta-theta quotients with simple poles}

\author{Giulia Cesana}
\author{Joshua Males}
\address{University of Cologne, Department of Mathematics and Computer Science, Division of Mathematics, Weyertal 86-90, 50931 Cologne, Germany}
\email{gcesana@math.uni-koeln.de}
\email{jmales@math.uni-koeln.de}

\begin{abstract}
	We employ a variant of Wright's circle method to determine the bivariate asymptotic behaviour of Fourier coefficients for a wide class of eta-theta quotients with simple poles in $\H$.
\end{abstract}

\maketitle

\section{Introduction}
Jacobi forms are ubiquitous throughout number theory and beyond. For example, they appear in string theory \cite{males2020asymptotic,RUSSO1996131}, the theory of black holes \cite{dabholkar2012quantum}, and the combinatorics of partition statistics \cite{bringmann2016dyson}. The Fourier coefficients of Jacobi forms often encode valuable arithmetic information. To describe a motivating example, let $\lambda$ be a partition of a positive integer $n$, i.e.\@ a list of non-increasing positive integers $\lambda_j$ with $1\leq j \leq s$ that sum to $n$. The \emph{rank} \cite{dyson} of $\lambda$ is given by the largest part minus the number of parts, and the \emph{crank} \cite{AndGa} of $\lambda$ is given by \begin{align*}
\begin{cases*}
\ell(\lambda) & \text{ if $\lambda$ contains no ones}, \\
\mu(\lambda) - \omega(\lambda) & \text{ if $\lambda$ contains ones}. 
\end{cases*}
\end{align*}
Here, $\ell(\lambda)$ denotes the largest part of $\lambda$, $\omega(\lambda)$ denotes the number of ones in $\lambda$, and $\mu(\lambda)$ denotes the number of parts greater than $\omega(\lambda)$. We denote by $M(m,n)$ (resp.\@ $N(m,n)$) the number of partitions of $n$ with crank $m$ (resp. rank $m$). It is well-known that the generating function of $M$ is given by (see \cite[equation (2.1)]{bringmann2016dyson})
\begin{align*}
 \sum_{\substack{n \geq 0 \\ m \in \Z}} M(m,n) \zeta^m q^n = \frac{i\left(\zeta^\frac 12-\zeta^{-\frac 12}\right) q^{\frac{1}{24}}\eta^2(\tau)}{\vartheta(z;\tau)},
\end{align*}
which is a Jacobi form and where $\zeta \coloneqq e^{2 \pi i z}$ for $z \in \C$, and $q \coloneqq e^{2 \pi i \tau}$ with $\tau \in \H$, the upper half-plane throughout. Here, the Dedekind $\eta$-function is given by
\begin{align*}
	\eta(\tau) \coloneqq q^{\frac{1}{24}} \prod_{n \geq 1} \left( 1-q^n \right),
\end{align*}
and the Jacobi theta function is defined by
\begin{equation*}
	\vartheta(z;\tau) \coloneqq i \zeta^{\frac{1}{2}} q^{\frac{1}{8}} \prod_{n\geq 1} (1-q^n)\left(1-\zeta q^n\right)\left(1-\zeta^{-1} q^{n-1}\right).
\end{equation*}

Note that a similar formula can be found for the generating function of $N$ as a mock Jacobi form involving an eta-theta quotient. In general Jacobi forms have a Fourier expansion of the form 
\begin{align} \label{Fourier expansion Jacobi form}
\sum_{\substack{n \geq 0 \\ m \in \Z}} a(m,n) \zeta^m q^n.
\end{align}
Classical results in the theory of modular forms (see \cite[Theorem 15.10]{bringmann2017harmonic} for example) give the asymptotic behaviour of the coefficients $a(m,n)$ of Jacobi forms similar to \eqref{Fourier expansion Jacobi form} for fixed $m$.

Many interesting examples of Jacobi forms arise as quotients of $\eta$ and $\vartheta$ functions. As an illuminating example, for $a_k, b_j \in \N$ and $n \in \Z$, consider the study of generalized theta quotients \cite[equation (13)]{thetablocks},
\begin{align*}
	\frac{\vartheta(a_1 z;\tau) \vartheta(a_2 z;\tau) \cdots \vartheta(a_k z:\tau)}{\vartheta(b_1 z;\tau) \vartheta(b_2 z;\tau) \cdots \vartheta(b_j z:\tau)} \eta(\tau)^n
\end{align*} 
which provide new constructions of Jacobi and Siegel modular forms. As highlighted by Gritsenko, Skoruppa, and Zagier, theta blocks also have deep applications to areas such as Fourier analysis over infinite-dimensional Lie algebras and the moduli spaces in algebraic geometry. In the present paper, we obtain the bivariate asymptotic behavior of the coefficients of a prototypical family of such generalized theta blocks, while the steps presented here also offer a pathway to obtain similar results for more general families.

In \cite{bringmann2016dyson} Bringmann and Dousse pioneered the use of new techniques in the study of the bivariate asymptotic behaviour of the Fourier coefficients and applied them to the partition crank function. In \cite{dousse2014asymptotic} Dousse and Mertens used these techniques to study the rank function. In particular, each of these papers used an extension of Wright's circle method \cite{wright1934asymptotic,wright1971stacks} to obtain bivariate asymptotics of  $N(m,n)$ and $M(m,n)$, with $m$ in a certain range depending on $n$.

Recently, the second author of the present paper extended these techniques to an example appearing in the partition function for entanglement entropy in string theory. In particular, \cite{males2020asymptotic,males2021asymptoticcorrigendum} considered the eta-theta quotient
\begin{equation}\label{eqn: entropy fn}
\frac{\vartheta(z;\tau)^4}{\eta(\tau)^9 \vartheta(2z;\tau)} =: \sum_{\substack{m \in \Z \\ n \geq 0}} b(m,n) \zeta^m q^m
\end{equation}
with a simple pole at $z= \frac{1}{2}$. Then bivariate asymptotic behaviour of the coefficients $b(m,n)$ is given by \cite[Theorem 1.1]{males2021asymptoticcorrigendum}.
\begin{theorem}
	For $\beta \coloneqq \pi \sqrt{\frac{2}{n}}$ and $|m| \leq \frac{1}{6 \beta} \log(n)$ we have that
	\begin{equation*}
	b(m,n) =(-1)^{m+ \delta + \frac{3}{2}} \frac{\beta^6 m}{  8 \pi^5  (2n)^{\frac{1}{4} }} e^{2 \pi \sqrt{2n}} + O \left( m n^{-\frac{15}{4}} e^{2 \pi \sqrt{2n}} \right)
	\end{equation*}
	as $n \rightarrow \infty$. Here, $\delta \coloneqq 1$ if $m <0$ and $\delta = 0$ otherwise.
\end{theorem}

The current paper serves to extend these results to a large family of eta-theta quotients with multiple simple poles\footnote{A similar framework exists for those without poles by simply extending the results of \cite{bringmann2016dyson,dousse2014asymptotic}.}. Such eta-theta quotients appear in numerous places. For example, investigations into Vafa-Witten invariants \cite[equation (2.5)]{alexandrov2020vafawitten} involve the functions
\begin{align}\label{eqn: VW}
\frac{i}{\eta(\tau)^{N-1} \vartheta(2z;\tau)},
\end{align}
which also appear in investigations into the counting of BPS-states via wall-crossing \cite[equation (5.114)]{Wot}. The asymptotics of this family of functions was studied in \cite{BriMa}. Other examples of similar shapes also arise as natural pieces of functions in investigations into BPS states, see e.g. \cite[Section 5.6.2 ]{Wot}.

Throughout, we consider an eta-theta quotient of the form
\begin{equation*}
f(z;\tau) \coloneqq \prod_{j=1}^{N} \eta(a_j \tau)^{\alpha_j} \frac{\vartheta( z;  \tau)}{\vartheta(b z; c \tau)},
\end{equation*}
where $a_j, b, c \in \N$ and $\alpha_j \in \Z$. Since we require asymptotic growth, we assume that $\sum_{j=1}^N\frac{\alpha_j}{a_j} < 0$. We omit the dependency on these parameters for notational ease. We assume that $b$ is even and $b^2>c$, and indicate the differences that would occur if $b$ were odd. In the language of \cite{thetablocks}, this is a family of generalized theta quotients.

\begin{remarks}
	\begin{enumerate}[wide, labelwidth=!, labelindent=0pt]
		\item The exposition presented here may be easily generalized to include products of theta functions in both the numerator and denominator of $f$, although this becomes lengthy to write out for the general case.
		
		\item We include a theta function in the numerator to allow us to assume that there are no poles of $f$ at the lattice points $0,1$. However, using the techniques presented here and shifting integrals to not have end-points at $0,1$ a similar story holds for functions without a theta function in the numerator.
	\end{enumerate}
\end{remarks}

We define the coefficients $c(m,n)$ by
\begin{equation*}
f(z;\tau) \eqqcolon \sum_{\substack{n \geq 0 \\ m \in \Z}} c(m,n) \zeta^m q^n,
\end{equation*}
for some $z$ in a small neighborhood of $0$ that is pole-free, and investigate their bivariate asymptotic behaviour. 
To this end, we employ and extend the techniques of \cite{bringmann2016dyson}, which also appear in \cite{dousse2014asymptotic,males2020asymptotic,males2021asymptoticcorrigendum}, using Wright's circle method to arrive at the following theorem. 

\begin{theorem}\label{Theorem: main}
	Define $\beta=\beta(n)\coloneqq\pi\sqrt{\frac{2}{n}}$, along with
	\begin{align*}
			\Lambda_1  &\coloneqq (-1)^{\frac 32+\sum\limits_{j=1}^{N}\alpha_j}  \left(\frac{1}{2\pi}\right)^{-\frac{1}{2}\sum\limits_{j=1}^{N}\alpha_j}  \frac{c^{\frac 12}}{4i\pi^2 \left(\frac{2b^2}{c}-1-\frac bc\right)} \prod\limits_{j=1}^{N}a_j^{-\frac{\alpha_j}{2}},
	\end{align*}
	and 
	\begin{align*}
		\Lambda_2  &\coloneqq  \frac{b^2}{c}-\frac bc+\frac{1}{4c}-\frac 14-\sum\limits_{j=1}^N\frac{\alpha_j}{12 a_j} .
	\end{align*}
	Assume that $0<1-\sum_{j=1}^{N} \frac{\alpha_j}{12 a_j }<\sqrt{\Lambda_2 }$, and $m=m(n)$ with $|m|\leq \frac{1}{6\beta}n^{-\delta} \log(n)$ for some small $\delta>0$ such that $m\to \infty$ as $n\to \infty$. Then
	\begin{align*}
		c(m,n) = \frac{-i}{2\pi} \Lambda_1  \beta^{2-\frac 12\sum\limits_{j=1}^N\alpha_j} \sqrt{\Lambda_2 }^{-\frac 12 \sum\limits_{j=1}^N \alpha_j } \frac{e^{2 \pi \sqrt{2\Lambda_2  n}}}{2 \pi \left(2\Lambda_2  n\right)^{\frac 14} } + O \left(\beta^{3-\frac 12\sum\limits_{j=1}^N\alpha_j} \frac{e^{2 \pi \sqrt{2\Lambda_2  n}}}{2 \pi \left(2\Lambda_2  n\right)^{\frac 14} }\right)
	\end{align*}
as $n\rightarrow\infty$. 
\end{theorem}

\begin{remark}
Note that the restriction on $\Lambda_2 $ still leaves infinitely many choices.
\end{remark}

The paper is structured as follows. We begin in Section \ref{Section: prelims} by recalling relevant results that are pertinent to the rest of the paper. Section \ref{Section: asymptotic behaviour of f} deals with defining the Fourier coefficients of $\zeta^m$ of $f$. In Section \ref{Section: towards dominant pole} we investigate the behaviour of $f$ toward the dominant pole $q = 1$. We follow this in Section \ref{Section: away from domminant pole} by bounding the contribution away from the pole at $q = 1$. In Section \ref{Section: Circle method} we obtain the asymptotic behaviour of $c(m,n)$ and hence prove Theorem \ref{Theorem: main}. 

\section*{acknowledgments}
The authors would like to thank Kathrin Bringmann, Andreas Mono, and Larry Rolen for many helpful comments on an earlier version of the paper. 


\section{Preliminaries}\label{Section: prelims}
Here we recall relevant definitions and results which will be used throughout the rest of the paper.

\subsection{Properties of $\vartheta$ and $\eta$}
When determining the asymptotic behaviour of $f$ we require the modular properties of both $\vartheta$ and $\eta$. It is well-known that $\vartheta$ is a Jacobi form (see e.g. \cite{mumford2007tata}).
\begin{lemma}\label{Lemma: transformation of theta}
	The function $\vartheta$ satisfies
	\begin{align*}
	\vartheta(z;\tau) = -\vartheta(-z ; \tau), \qquad \vartheta(z;\tau) = -\vartheta(z+1;\tau),\qquad \vartheta(z; \tau) = \frac{i}{\sqrt{-i \tau}} e^{ \frac{- \pi i z^2}{\tau}} \vartheta\left( \frac{z}{\tau} ; -\frac{1}{\tau} \right).
	\end{align*}
\end{lemma}
We also have the well-known triple product formula (see e.g. \cite[Proposition 1.3]{zwegers2008mock}), yielding
\begin{equation}\label{Equation: summation rep for theta}
\vartheta(z;\tau) = iq^{\frac{1}{8}} \zeta^{\frac{1}{2}} \sum_{n \in \Z} (-1)^n q^{\frac{n^2+n}{2}} \zeta^{n}.
\end{equation}
Furthermore, we have the following modular transformation formula of $\eta$ (see e.g. \cite{koecher2007elliptische}).

\begin{lemma}\label{Lemma: transformation of eta}
	We have that
	\begin{equation*}
	\begin{split}
\eta(\tau) = \sqrt{\frac{i}{\tau}} 	\eta\left( -\frac{1}{\tau} \right) .
	\end{split}
	\end{equation*}
\end{lemma}

\subsection{Integrals over segments of circles} 
Let $U_r(z_0)\coloneqq\{z: |z-z_0|<r\}$ be the open disc around $z_0 \in \C$ with radius $r$. Then we have the following result \cite[page 263]{curtiss1978introduction}. 

\begin{lemma}  \label{Theorem: integral over segment of circle}
	Let $g: U_r(z_0)\backslash\{z_0\} \rightarrow\mathbb{C}$ be analytic and have a simple pole at $z_0$. Let $\gamma(\delta)$ be a circular arc with parametric equation $z=z_0+\delta e^{i\theta}$, for $-\pi<\theta_1\leq\theta\leq\theta_2\leq \pi$ and $0<\delta< r$. Then
	$$ \lim_{\delta\rightarrow 0} \int_{\gamma(\delta)} g(z) dz =  i (\theta_2-\theta_1) \operatorname{Res}(g,z_0).$$
\end{lemma}

\subsection{A particular bound}
We require a bound on the size of 
\begin{equation*}
P(q) \coloneqq \frac{q^{\frac{1}{24}}}{\eta(\tau)}, 
\end{equation*}
away from the pole at $q=1$. For this we use \cite[Lemma 3.5]{bringmann2016dyson}.
\begin{lemma}
	Let $\tau = u +iv \in \H$ with $Mv \leq u \leq \frac{1}{2}$ for $u>0$ and $v \rightarrow 0$. Then
	\begin{equation*}
	|P(q)| \ll \sqrt{v} \exp \left[ \frac{1}{v} \left(\frac{\pi}{12} - \frac{1}{2\pi} \left(1- \frac{1}{\sqrt{1+M^2}}\right)\right) \right].
	\end{equation*}
\end{lemma}
In particular, with $v = \frac{\beta}{2\pi}$, $u=\frac{\beta m^{-\frac 13} x}{2 \pi}$ and $M = m^{-\frac 13}$ this gives for $1 \leq x \leq \frac{\pi m^{\frac 13}}{\beta}$ the bound
\begin{equation}\label{Equation: bound on P(q)}
|P(q)| \ll n^{-\frac{1}{4}} \exp \left[ \frac{2 \pi}{\beta} \left( \frac{\pi}{12} - \frac{1}{2\pi} \left( 1- \frac{1}{\sqrt{1 + m^{-\frac 23}}} \right) \right) \right].
\end{equation}

\subsection{$I$-Bessel functions}
Here we recall relevant results on the $I$-Bessel function which may be written as
\begin{equation*}
I_{\ell} (x) \coloneqq \frac{1}{2 \pi i} \int_{\Gamma} t^{-\ell - 1}e^{\frac{x}{2} \left(t+\frac{1}{t} \right)} dt,
\end{equation*}
where $\Gamma$ is a contour which starts in the lower half plane at $-\infty$, surrounds the origin counterclockwise and returns to $-\infty$ in the upper half-plane. We are particularly interested in the asymptotic behaviour of $I_\ell$, given in the following lemma (see e.g. \cite[(4.12.7)]{andrews1999special}).
\begin{lemma}\label{Lemma: asymptotic of I Bessel}
	For fixed $\ell$ we have
	\begin{equation*}
	I_\ell(x) = \frac{e^x}{\sqrt{2 \pi x}} + O\left(\frac{e^x}{x^{\frac{3}{2}}}\right)
	\end{equation*}
	as $x \rightarrow \infty$.
\end{lemma}

\section{Fourier Coefficients of $f$}\label{Section: asymptotic behaviour of f}

 Note that $f(-z;\tau)=f(z;\tau)$ by Lemma \ref{Lemma: transformation of theta}, and so $c(-m,n)=c(m,n)$. For the case $m=0$ one can use classical results (see e.g. \cite[Theorem 15.10]{bringmann2017harmonic}) to calculate the Fourier coefficients. We therefore restrict our attention to the case $m> 0$.

We first define the Fourier coefficients of $\zeta^m$ of $f$. Since we focus only on the case $z \in [0,1]$, we let $h_1,...,h_s\in\mathbb{Q}$ denote the poles of $f$ in this range. Note that the distribution of the poles is symmetric on the interval in question. 

Define the path of integration $\Gamma_{\ell,r}$ by
\begin{align*}
	\Gamma_{\ell,r} \coloneqq 
	\begin{dcases}
		0 \text{ to } h_1-r & \text{ if } \ell=0,\\
		h_{\ell}+r \text{ to } h_{\ell+1}-r & \text{ if } 1\leq\ell\leq s-1,\\
		h_s+r \text{ to } 1 & \text{ if } \ell=s,\\
	\end{dcases}
\end{align*}
for some $r>0$ sufficiently small. Following the framework of \cite{dabholkar2012quantum,males2020asymptotic,males2021asymptoticcorrigendum}, we define 
\begin{align*}
f_m^{\pm} (\tau) &\coloneqq\sum\limits_{\ell=0}^{s}  \int_{\Gamma_{\ell,r}}f(z;\tau) e^{-2\pi i mz} dz +\sum\limits_{\ell=1}^{s}G^{\pm}_{\ell,r},
\end{align*}
where
\begin{align*}
G^{\pm}_{\ell,r} \coloneqq \int\limits_{\gamma_{\ell,r}^{\pm}} f(z;\tau)e^{-2\pi imz} dz
\end{align*}
for a fixed pole $h_\ell$ ($1\leq \ell\leq s$).
Here, $\gamma_{\ell,r}^+$ is the semi-circular path of radius $r$ passing above the pole $h_\ell$ and $\gamma_{\ell,r}^-$ is the semi-circular path passing below the pole $h_\ell$, see Figures \ref{pathG+} and \ref{pathGl+}.

\begin{figure}[h!] 
	\centering
	\begin{tikzpicture}

	\begin{scope}[thick]
	\draw                  (0,0) circle (.5pt) node[left] {$0$} -- (0.5,0); 
	
	\draw[->] (0.5,0)  arc (180:90:.5cm); 		
	\draw[] (1.5,0) arc (0:90:0.5cm);			
	
	\draw[] (1,0) circle (.5pt) node[below] {$h_1$};
	
	\draw                  (1.5,0) -- (2.25,0);  
	
	\draw[->] (2.25,0)  arc (180:90:.5cm); 	
	\draw[] (3.25,0) arc (0:90:0.5cm);		
	
	\draw[] (2.75,0) circle (.5pt) node[below] {$h_2$};
	
	\draw                  (3.25,0) -- (4,0);
	
	\draw[->] (4,0)  arc (180:90:.5cm); 	
	\draw[] (5,0) arc (0:90:0.5cm);
	
	\draw[] (4.5,0) circle (.5pt) node[below] {$h_3$};
	
	\draw                  (5,0) -- (5.75,0);
	
	\draw[dotted]         (5.75,0)  --  (7,0);   
	
	\draw                  (7,0) -- (7.75,0);
	
	\draw[->] (7.75,0)  arc (180:90:.5cm); 	
	\draw[] (8.75,0) arc (0:90:0.5cm);
	
	\draw[] (8.25,0) circle (.5pt) node[below] {$h_{s-2}$};
	
	\draw                  (8.75,0) -- (9.5,0);
	
	\draw[->] (9.5,0)  arc (180:90:.5cm); 	
	\draw[] (10.5,0) arc (0:90:0.5cm);
	
	\draw[] (10,0) circle (.5pt) node[below] {$h_{s-1}$};
	
	\draw                  (10.5,0) -- (11.25,0);
	
	\draw[->] (11.25,0)  arc (180:90:.5cm); 	
	\draw[] (12.25,0) arc (0:90:0.5cm);
	
	\draw[] (11.75,0) circle (.5pt) node[below] {$h_{s}$};
	
	\draw         (12.25,0) -- (12.75,0)  circle (.5pt) node[right] {$1$};
	\end{scope}
	\end{tikzpicture}
	\caption{The path of integration taking $\gamma^+_{\ell,r}$ at each pole.} 
	\label{pathG+}
\end{figure}
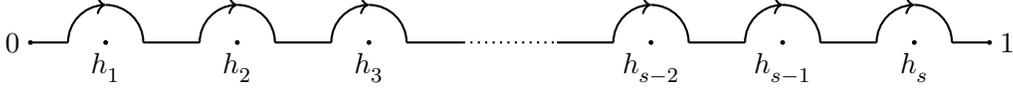

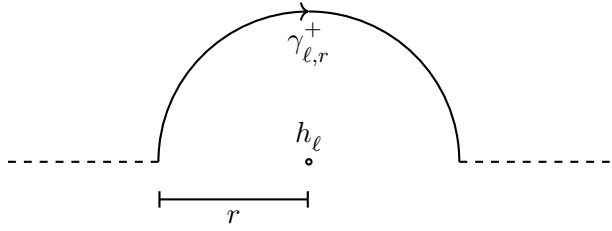
\begin{figure}[h!]
	\centering
	\begin{tikzpicture}  
	
	\begin{scope}[thick]
	
	\draw[dashed]         (-6,0)  --  (-4,0);   
	
	\draw[->] (-4,0) arc (180:90:2cm) node[below] {$\gamma_{\ell,r}^+$};  
	\draw[] (0,0) arc (0:90:2cm);		
	
	\draw[] (-2,0) circle (1pt) node[above] {$h_\ell$};
	
	\draw[|- |] (-4,-0.5) --node[below] {$r$} (-2,-0.5) ;
	
	\draw[dashed]         (0,0)  --  (2,0);
	
	\end{scope}
	
	\end{tikzpicture}
	\caption{The contour $\gamma^+_{\ell,r}$ for a fixed $\ell$.}
	\label{pathGl+}
\end{figure}
Following \cite{dabholkar2012quantum} the Fourier coefficient of $\zeta^m$ of $f$, for fixed $m$, is given by 
\begin{align}\label{Equation: definition of f_m}
f_m(\tau) \coloneqq& \lim\limits_{r\rightarrow 0^+} \frac{f^+_m(\tau)+f^-_m(\tau)}{2} 
=\lim\limits_{r\rightarrow 0^+} \frac 12 \left(2\sum\limits_{\ell=0}^{s}  \int_{\Gamma_{\ell,r}}f(z;\tau) e^{-2\pi i mz} dz+ \sum_{\ell=1}^{s} G^+_{\ell,r}+G^-_{\ell,r}\rule{0pt}{1cm}\right).  
\end{align}
For fixed $\ell$ we use Lemma \ref{Theorem: integral over segment of circle} to see that 
$$ \lim\limits_{r\rightarrow 0^+} \left(G^+_{\ell,r}+G^-_{\ell,r}\right) = 0, $$
since we only have simple poles.

The substitution $z\mapsto 1-z$ gives us
	\begin{align*}
		\sum\limits_{\ell=0}^{s} \int_{\Gamma_{\ell,r}}f(z;\tau) e^{-2\pi i mz} dz=&  -\sum\limits_{\ell=0}^{s}\int_{\Gamma_{\ell,r}}f(z;\tau) e^{2\pi i m z} dz,
	\end{align*}
	since $b$ is even and using that $f(1-z;\tau)=(-1)^{b+1} f(z;\tau)$ by Lemma \ref{Lemma: transformation of theta}.
	Thus, \eqref{Equation: definition of f_m} simplifies to
	\begin{align}\label{f_m with cancelled residue part}
		f_m(\tau) =& -i \lim_{r\rightarrow 0^+} \sum_{\ell=0}^s \int_{\Gamma_{\ell,r}} f(z;\tau) \sin(2\pi mz) dz .
	\end{align}

\begin{remark}
	For odd $b$ one would obtain a similar formula with the integrand replaced by \newline$f(z;\tau) \cos(2\pi mz)$.
\end{remark}

In the following two sections we determine the asymptotic behavior of $f$ towards and away from the dominant pole at $q=1$, respectively. From now on we will let $\tau=\frac{i\varepsilon}{2\pi}$, $\varepsilon \coloneqq \beta\left(1+ix m^{-\frac{1}{3}}\right)$, $\beta \coloneqq \pi\sqrt{\frac 2n}$ and $|m|\leq \frac{1}{6\beta}n^{-\delta} \log(n)$ for some small $\delta>0$ such that $m\to \infty$ as $n\to \infty$.

\section{Bounds toward the dominant pole}\label{Section: towards dominant pole}

In this section we consider the behavior of $f_m$ toward the dominant pole at $q=1$. 
\begin{lemma} \label{Lemma: asymptotic behaviour f towards the dominant pole}
	Let $\tau=\frac{i\varepsilon}{2\pi}$, with $0<\operatorname{Re}(\varepsilon)\ll1$, let $z$ be away from the poles, let $\mathcal{M}(z)$ be an explicit function which is positive for all $z\in(0,1)$, and let
	$$ \mathcal{C}\left(z,\tau\right) \coloneqq (-1)^{\sum\limits_{j=1}^{N}\alpha_j}  \left(\frac{\varepsilon}{2\pi}\right)^{-\frac{1}{2}\sum\limits_{j=1}^{N}\alpha_j} c^{\frac 12}\left(\prod\limits_{j=1}^{N}a_j^{-\frac{\alpha_j}{2}}\right)  \frac{\sinh\left(\frac{2\pi^2 z}{\varepsilon}\right)}{ \sinh\left(\frac{2\pi^2 b z}{c\varepsilon}\right)} e^{\frac{2\pi^2}{\varepsilon} \left(\frac{4b^2z^2+1}{4c}-\frac{4z^2+1}{4}-\sum\limits_{j=1}^N\frac{\alpha_j}{12 a_j}\right)}.$$
	Then we have that
	\begin{align*}
	f\left(z;\frac{i \varepsilon}{2 \pi}\right) = \mathcal{C}\left(z,\frac{i\varepsilon}{2\pi}\right) \left(1 +O\left(e^{-\frac{4\pi^2}{\varepsilon}\mathcal{M}(z)}\right)\right)
	\end{align*}
	as $n \rightarrow \infty$.
\end{lemma}
\begin{proof}
	Using Lemmas \ref{Lemma: transformation of theta} and \ref{Lemma: transformation of eta} and the definitions of $\eta$ and $\vartheta$ a lengthy but straightforward calculation shows that 
	\begin{align} 
	 f\left(z;\frac{i\varepsilon}{2\pi}\right)=\mathcal{C}\left(z,\frac{i\varepsilon}{2\pi}\right) \prod\limits_{\kappa \geq1} \frac{\left(1-e^{-\frac{4\pi^2\kappa }{\varepsilon}}\right)\left(1-e^{\frac{4\pi^2}{\varepsilon}\left(z-\kappa \right)}\right)\left(1-e^{\frac{4\pi^2}{\varepsilon}\left(-z-\kappa \right)}\right)\prod\limits_{j=1}^N\left(1-e^{-\frac{4\pi^2\kappa }{a_j\varepsilon}}\right)^{\alpha_j}}{\left(1-e^{-\frac{4\pi^2\kappa }{c\varepsilon}}\right)\left(1-e^{\frac{4\pi^2}{c\varepsilon}\left(b z-\kappa \right)}\right)\left(1-e^{\frac{4\pi^2}{c\varepsilon}\left(-b z-\kappa \right)}\right)}. \notag
	\end{align}

	In order to find a bound we inspect the asymptotic behavior of the product over $\kappa$. Splitting $\alpha_j$ into positive and negative powers, labeled by $\gamma_j,\delta_j\in\mathbb{N}$, and $a_j$ into $x_j$ and $y_j$ respectively we first rewrite this as
	
	\makeatletter
	\renewcommand{\maketag@@@}[1]{\hbox{\m@th\normalsize\normalfont#1}}
	\makeatother
	
	\begin{small}
	\begin{align}\label{equation: main product}
	\prod\limits_{\kappa \geq1} \frac{\left(1-e^{-\frac{4\pi^2\kappa }{\varepsilon}}\right)\left(1-e^{-\frac{4\pi^2}{\varepsilon}\left(\kappa -z\right)}\right)\left(1-e^{-\frac{4\pi^2}{\varepsilon}\left(\kappa +z\right)}\right)\prod\limits_{j=1}^{N_1} \left(1-e^{-\frac{4\pi^2 \kappa }{x_j\varepsilon}}\right)^{\gamma_j}  \prod\limits_{k=1}^{N_2}\left(\sum\limits_{\mu\geq 0} e^{-\frac{4\pi^2\mu \kappa }{y_k\varepsilon}}\right)^{\delta_k}}{\left(1-e^{-\frac{4\pi^2\kappa }{c\varepsilon}}\right)\left(1-e^{-\frac{4\pi^2}{c\varepsilon}\left(\kappa -bz\right)}\right)\left(1-e^{-\frac{4\pi^2}{c\varepsilon}\left(\kappa +bz\right)}\right)},
	\end{align}
	\end{small}
$\!\!$since $| e^{-\frac{4\pi^2\kappa}{y_k\varepsilon}}|<1 $ for all $\kappa\geq1$. We also have that $| e^{-\frac{4\pi^2\kappa }{c\varepsilon}}|<1$ and $| e^{-\frac{4\pi^2 }{c\varepsilon}(\kappa + bz)}|<1$ for all $\kappa\geq1$ since $b,c\in\N$. Therefore, we have that
	\begin{align*}
	\frac{1}{\left(1-e^{-\frac{4\pi^2\kappa }{c\varepsilon}}\right)\left(1-e^{-\frac{4\pi^2}{c\varepsilon}\left(\kappa +bz\right)}\right)} =\sum_{\lambda\geq0} e^{-\frac{4\pi^2\lambda \kappa }{c\varepsilon}} \sum_{\xi\geq0} e^{-\frac{4\pi^2\xi}{c\varepsilon}\left(\kappa +bz\right)} .
	\end{align*}
	
	Up to this point our calculations are independent of the size of $z$. The remaining term is
	\begin{align*} 
	\frac{1}{1-e^{-\frac{4\pi^2}{c\varepsilon}\left(\kappa-bz\right)}}.
	\end{align*}
	Let $\kappa _0$ be the smallest $\kappa \geq1$ such that $(\kappa -bz)\geq0$. We may rewrite
	\begin{align*}
	\prod_{\kappa  \geq 1} \frac{1}{\left(1-e^{-\frac{4\pi^2}{c\varepsilon}\left(\kappa -bz\right)}\right)}
	&=\prod_{\kappa =1}^{\kappa _0-1} \frac{1}{\left(1-e^{-\frac{4\pi^2}{c\varepsilon}\left(\kappa -bz\right)}\right)}\prod_{\kappa \geq \kappa _0} \sum_{\mu\geq0} e^{-\frac{4\pi^2\mu}{c\varepsilon}\left(\kappa -bz\right)}.
	\end{align*}
	The first product is 
	\begin{align*}
	 \prod_{\kappa =1}^{\kappa _0-1}\frac{1}{\left(1-e^{-\frac{4\pi^2}{c\varepsilon}\left(\kappa -bz\right)}\right)} = \prod_{\kappa =1}^{\kappa _0-1} -e^{\frac{4\pi^2}{c\varepsilon}\left(\kappa -bz\right)}\sum_{\nu\geq0}e^{\frac{4\pi^2\nu}{c\varepsilon}\left(\kappa -bz\right)}.
	\end{align*}
	
	Let 
	\begin{align*}
	\mathcal{M}(z)\coloneqq \begin{cases}
	\min\left(1-z, \frac{1}{x_j}, \frac{1}{y_k}, \frac{1}{c}, \frac{\kappa_0-bz}{c}, \frac{bz+1-\kappa_0}{c} \right), & \text{ if } \kappa_0 \neq 1, \\
	\min\left(1-z, \frac{1}{x_j}, \frac{1}{y_k}, \frac{1}{c}, \frac{\kappa_0-bz}{c} \right), & \text{ if } \kappa_0 = 1,
	\end{cases} 
	\end{align*}
	running over all $1\leq j\leq N_1$ and $1\leq k\leq N_2$. Note that for $0<\operatorname{Re}(\varepsilon)\ll 1$, and $z\in (0,1)$ we have $\mathcal{M}(z)>0$, so the product in \eqref{equation: main product} is of order
	\begin{align*}
	1 +O\left(e^{-\frac{4\pi^2}{\varepsilon}\mathcal{M}(z)}\right),
	\end{align*}
	which finishes the proof.
\end{proof}

\begin{remark}
	By separating into cases, one is able to obtain more precise asymptotics. However, this is not required for the sequel and we leave the details for the interested reader.
\end{remark}

\begin{theorem} \label{Theorem: bound f_m near pole}
	For $|x|\leq 1$ we have that
	\begin{align*}
	f_m\left(\frac{i\varepsilon}{2\pi}\right) =& \,\Lambda_1 \, \varepsilon^{1-\frac{1}{2}\sum\limits_{j=1}^{N}\alpha_j} e^{ \frac{2\pi^2}{\varepsilon}\Lambda_2 }    +O\left( \beta^{2-\frac{1}{2}\sum\limits_{j=1}^{N}\alpha_j}   e^{ \frac{2\pi^2}{\varepsilon} \Lambda_2 }\right)
	\end{align*}
as $n\rightarrow\infty$.
\end{theorem}

\begin{proof}
	Plugging Lemma \ref{Lemma: asymptotic behaviour f towards the dominant pole} into \eqref{f_m with cancelled residue part} yields
	\begin{align}\label{equation: Fourier coefficients plugging in f}
			f_m\left(\frac{i\varepsilon}{2\pi}\right) 
			&= -i\sum_{\ell=0}^{s} \lim_{r\rightarrow 0^+}  \int_{\Gamma_{\ell,r}} \mathcal{C}\left(z,\frac{i\varepsilon}{2\pi}\right) \left(1 +O\left(e^{-\frac{4\pi^2}{\varepsilon}\mathcal{M}(z)}\right)\right) \sin(2\pi mz) dz.
	\end{align}
We have that
\begin{align*}
	\frac{\sinh\left(\frac{2\pi^2 z}{\varepsilon}\right)}{ \sinh\left(\frac{2\pi^2 b z}{c\varepsilon}\right)} =  e^{\frac{2\pi^2 }{\varepsilon}z\left(1-\frac bc\right)}\left(1-e^{-\frac{4\pi^2 z}{\varepsilon}}\right) \sum_{\lambda\geq0} e^{-\frac{4\pi^2\lambda bz}{c\varepsilon}} = e^{\frac{2\pi^2 }{\varepsilon}z\left(1-\frac bc\right)}\left(1+O\left(e^{-\frac{4\pi^2 z}{\varepsilon}}\right)\right),
\end{align*}
using $|e^{-\frac{4\pi^2 bz}{c\varepsilon}}|<1$. Additionally we see that
$$ e^{\frac{2\pi^2}{\varepsilon} \left(\frac{4b^2z^2+1}{4c}-\frac{4z^2+1}{4}-\sum\limits_{j=1}^N\frac{\alpha_j}{12 a_j}\right)} = e^{\frac{2\pi^2}{\varepsilon} \left(\frac{b^2}{c}-1\right)z^2}\; e^{\frac{2\pi^2}{\varepsilon}\left(\frac{1}{4c}-\frac 14-\sum\limits_{j=1}^N\frac{\alpha_j}{12 a_j}\right)}  .$$
Defining 
$$ \Omega(m,n) \coloneqq (-1)^{\sum\limits_{j=1}^{N}\alpha_j}  \left(\frac{\varepsilon}{2\pi}\right)^{-\frac{1}{2}\sum\limits_{j=1}^{N}\alpha_j} c^{\frac 12}\left(\prod\limits_{j=1}^{N}a_j^{-\frac{\alpha_j}{2}}\right) e^{\frac{2\pi^2}{\varepsilon} \left(\frac{1}{4c}-\frac 14-\sum\limits_{j=1}^N\frac{\alpha_j}{12 a_j}\right)} $$
we can therefore rewrite \eqref{equation: Fourier coefficients plugging in f} as
\begin{align*} 
	 -i\Omega(m,n)\sum_{\ell=0}^{s} \lim_{r\rightarrow 0^+} \int_{\Gamma_{\ell,r}}  e^{\frac{2\pi^2}{\varepsilon} \left(\frac{b^2}{c}-1\right)z^2}\; e^{\frac{2\pi^2}{\varepsilon} \left(1-\frac bc\right)z}  \left(1 +O\left(e^{-\frac{4\pi^2}{\varepsilon}\mathcal{N}(z)}\right)\right)  \sin(2\pi mz) dz,
\end{align*}
where $\mathcal{N}(z)\coloneqq\min(z,\mathcal{M}(z))$. 

We immediately see that this splits up into two integrals
\begin{align} \label{fm integral main part}
	\sum_{\ell=0}^{s} \lim_{r\rightarrow 0^+}\int_{\Gamma_{\ell,r}}  e^{\frac{2\pi^2}{\varepsilon} \left(\frac{b^2}{c}-1\right)z^2}\; e^{\frac{2\pi^2}{\varepsilon} \left(1-\frac bc\right)z}  \sin(2\pi mz) dz
\end{align} 
and 
\begin{align} \label{fm integral error part}
	\sum_{\ell=0}^{s} \lim_{r\rightarrow 0^+}\int_{\Gamma_{\ell,r}}  e^{\frac{2\pi^2}{\varepsilon} \left(\frac{b^2}{c}-1\right)z^2}\; e^{\frac{2\pi^2}{\varepsilon} \left(1-\frac bc\right)z}  O\left(e^{-\frac{4\pi^2}{\varepsilon}\mathcal{N}(z)}\right)  \sin(2\pi mz) dz.
\end{align}

For arbitrary $\mathcal{H}_1,\mathcal{H}_2\in\mathbb{C}$, with $\mathcal{H}_2\neq0$ the following formula holds (which may be checked by a computer algebra package, e.g. MAPLE)
\begin{align*}
	\int_{t}^u &e^{\mathcal{H}_1z} e^{\mathcal{H}_2z^2}\sin(2\pi mz) dz \notag \\
	=& -\frac{1}{\sqrt{-\mathcal{H}_2}} \left(\frac 14 i\sqrt{\pi} \left(e^{-\frac 14 \frac{(\mathcal{H}_1+2\pi im)^2}{\mathcal{H}_2}}\operatorname{erf}\left(\frac 12 \frac{2\mathcal{H}_2t+\mathcal{H}_1+2\pi im}{\sqrt{-\mathcal{H}_2}}\right)\right.\right. \notag \\
	& \left.\left. +e^{-\frac 14 \frac{(-\mathcal{H}_1+2\pi im)^2}{\mathcal{H}_2}} \operatorname{erf}\left(\frac 12 \frac{-2 \mathcal{H}_2 t - \mathcal{H}_1 + 2\pi i m}{\sqrt{-\mathcal{H}_2}}\right) - e^{-\frac 14 \frac{(\mathcal{H}_1+2\pi im)^2}{\mathcal{H}_2}} \operatorname{erf}\left(\frac 12 \frac{2\mathcal{H}_2u + \mathcal{H}_1+2\pi im}{\sqrt{-\mathcal{H}_2}}\right)\right.\right. \notag \\
	& \left.\left.-e^{-\frac 14 \frac{(-\mathcal{H}_1+2\pi im)^2}{\mathcal{H}_2}} \operatorname{erf}\left(\frac 12 \frac {-2\mathcal{H}_2u-\mathcal{H}_1+2\pi im}{\sqrt{-\mathcal{H}_2}}\right)\right)\right),
\end{align*}
where $\operatorname{erf} (z)\coloneqq\frac{2}{\sqrt{\pi}} \int_{0}^{z} e^{-t^2}dt$ denotes the error function.
We thus obtain 
\begin{align} \label{equation: Maple calculation with cancellations}
	& \sum_{\ell=0}^{s} \lim_{r\rightarrow 0^+}\int_{\Gamma_{\ell,r}}  e^{\mathcal{H}_2z^2}\; e^{\mathcal{H}_1z}  \sin(2\pi mz) dz \notag \\
	=& -\frac{1}{\sqrt{-\mathcal{H}_2}} \left(\frac 14 i\sqrt{\pi} \left(e^{-\frac 14 \frac{(\mathcal{H}_1+2\pi im)^2}{\mathcal{H}_2}}\operatorname{erf}\left(\frac 12 \frac{\mathcal{H}_1+2\pi im}{\sqrt{-\mathcal{H}_2}}\right) +e^{-\frac 14 \frac{(-\mathcal{H}_1+2\pi im)^2}{\mathcal{H}_2}} \operatorname{erf}\left(\frac 12 \frac{ - \mathcal{H}_1 + 2\pi i m}{\sqrt{-\mathcal{H}_2}}\right)\right.\right. \notag \\
	& \left.\left.  - e^{-\frac 14 \frac{(\mathcal{H}_1+2\pi im)^2}{\mathcal{H}_2}} \operatorname{erf}\left(\frac 12 \frac{2\mathcal{H}_2  + \mathcal{H}_1+2\pi im}{\sqrt{-\mathcal{H}_2}}\right)-e^{-\frac 14 \frac{(-\mathcal{H}_1+2\pi im)^2}{\mathcal{H}_2}} \operatorname{erf}\left(\frac 12 \frac {-2\mathcal{H}_2-\mathcal{H}_1+2\pi im}{\sqrt{-\mathcal{H}_2}}\right)\right)\right),
\end{align}
since all the other terms cancel. If $|\operatorname{Arg}(\pm z)| <\frac{\pi}{4}$, we have that (see e.g. \cite{bringmann2019framework})
\begin{align}\label{asymptotic behaviour error function}
	\operatorname{erf}\left(iz\right) = \frac{ie^{z^2}}{\sqrt{\pi}z}\left(1+O\left(|z|^{-2}\right)\right) = \frac{ie^{z^2}}{\sqrt{\pi}z}+O\left(e^{z^2}|z|^{-3}\right), 
\end{align}
 as $|z|\rightarrow\infty$. 

Consider the integral \eqref{fm integral main part}. In this case, since $\frac{b^2}{c}>1$, we obtain
\begin{align*}
	\frac 12 \frac{\pm\mathcal{H}_1+2\pi im}{\sqrt{-\mathcal{H}_2}} =  \frac{\pm\left(\frac{\pi}{\varepsilon} \left(1-\frac bc\right)\right)+im}{i\sqrt{\frac{2}{\varepsilon} \left(\frac{b^2}{c}-1\right)}}.
\end{align*}
 We have $\varepsilon=\beta(1+ixm^{-\frac 13})$, and since $m \rightarrow \infty$ as $n \rightarrow \infty$ we see that $\varepsilon \rightarrow \beta $ in this limit.
This yields
\begin{align*}
	\lim_{n\to\infty} \frac 12 \frac{\pm\mathcal{H}_1+2\pi im}{\sqrt{-\mathcal{H}_2}} = i \frac{\mp\left(\frac{\pi}{\beta} \left(1-\frac bc\right)\right)-im}{\sqrt{\frac{2}{\beta} \left(\frac{b^2}{c}-1\right)}}
\end{align*}
and we see that 
$$\left|\pm\frac{\pi}{\beta} \left(1-\frac bc\right) \right| > \frac{1}{6\pi} \left(\frac n2\right)^{\frac 12 -\delta}\log(n)\geq |m|,$$ 
for some small $\delta>0$.  

Furthermore, we have 
\begin{align*}
	\frac 12 \frac {\pm2\mathcal{H}_2\pm\mathcal{H}_1+2\pi im}{\sqrt{-\mathcal{H}_2}} =\frac {\pm2\left(\frac{\pi}{\varepsilon} \left(\frac{b^2}{c}-1\right)\right)\pm\left(\frac{\pi}{\varepsilon} \left(1-\frac bc\right)\right)+ im}{i\sqrt{\frac{2}{\varepsilon} \left(\frac{b^2}{c}-1\right)}}.
\end{align*}
Using the same argument as before we obtain
\begin{align*}
	\lim_{n\to\infty} \frac 12 \frac {\pm2\mathcal{H}_2\pm\mathcal{H}_1+2\pi im}{\sqrt{-\mathcal{H}_2}}  = i \frac {\mp2\left(\frac{\pi}{\beta} \left(\frac{b^2}{c}-1\right)\right)\mp\left(\frac{\pi}{\beta} \left(1-\frac bc\right)\right)- im}{\sqrt{\frac{2}{\beta} \left(\frac{b^2}{c}-1\right)}} 
\end{align*}
and see that
\begin{align*}
	\left| \mp2\left(\frac{\pi}{\beta} \left(\frac{b^2}{c}-1\right)\right)\mp\left(\frac{\pi}{\beta} \left(1-\frac bc\right)\right) \right| > \frac{1}{6\pi} \left(\frac n2\right)^{\frac 12 -\delta}\log(n)\geq |m|.
\end{align*}
Therefore the arguments of the error functions in \eqref{equation: Maple calculation with cancellations} satisfy the condition of \eqref{asymptotic behaviour error function}. Plugging in yields
\begin{small}
	\begin{align*}
		\sum_{\ell=0}^{s} & \lim_{r\rightarrow 0^+}\int_{\Gamma_{\ell,r}}  e^{\mathcal{H}_2z^2}\; e^{\mathcal{H}_1z}  \sin(2\pi mz) dz \notag \\
		=& -\frac{1}{\sqrt{-\mathcal{H}_2}} \left(\frac 14 i\sqrt{\pi}\left(- \frac{ie^{ \mathcal{H}_2+\mathcal{H}_1}}{\sqrt{\pi}\left(-i\frac 12 \frac{2\mathcal{H}_2+\mathcal{H}_1+2\pi im}{\sqrt{-\mathcal{H}_2}}\right)}+O\left(e^{ \mathcal{H}_2+\mathcal{H}_1}\left|i\frac 12 \frac{2\mathcal{H}_2+\mathcal{H}_1+2\pi im}{\sqrt{-\mathcal{H}_2}}\right|^{-3}\right)\right) \right) \notag \\
		=&  \frac{e^{ \mathcal{H}_2+\mathcal{H}_1}}{4i \mathcal{H}_2+2i\mathcal{H}_1-4\pi m}+O\left(-\frac 14 i\sqrt{\pi}\frac{e^{ \mathcal{H}_2+\mathcal{H}_1}}{\sqrt{-\mathcal{H}_2}}\left|i\frac 12 \frac{2\mathcal{H}_2+\mathcal{H}_1+2\pi im}{\sqrt{-\mathcal{H}_2}}\right|^{-3}\right).
	\end{align*}
\end{small}
We therefore obtain that \eqref{fm integral main part} equals 
\begin{align*}
	 \frac{e^{ \frac{2\pi^2}{\varepsilon} \left(\frac{b^2}{c}-\frac bc\right)}}{\frac{4i\pi^2}{\varepsilon} \left(\frac{2b^2}{c}-1-\frac bc\right)-4\pi m}  +O\left(-\frac{\sqrt{\pi}}{4}\frac{e^{ \frac{2\pi^2}{\varepsilon} \left(\frac{b^2}{c}-\frac bc\right)}}{\sqrt{\frac{2\pi^2}{\varepsilon} \left(\frac{b^2}{c}-1\right)}}\left|\frac{\frac{\pi^2}{\varepsilon} \left(\frac{2b^2}{c}-1-\frac bc\right)+\pi im}{\sqrt{\frac{2\pi^2}{\varepsilon} \left(\frac{b^2}{c}-1\right)}}\right|^{-3}\right).
\end{align*}
 Combining this, along with the fact that $\mathcal{N}(z)>0$ and recycling the same arguments for \eqref{fm integral error part}, yields
\begin{align*}
	f_m\left(\frac{i\varepsilon}{2\pi}\right) =&-i\Omega(m,n)\left( \frac{e^{ \frac{2\pi^2}{\varepsilon} \left(\frac{b^2}{c}-\frac bc\right)}}{\frac{4i\pi^2}{\varepsilon} \left(\frac{2b^2}{c}-1-\frac bc\right)-4\pi m}  \right. \\
	&\quad \quad \quad \quad \quad \quad \left.+O\left(-\frac{\sqrt{\pi}}{4}\frac{e^{ \frac{2\pi^2}{\varepsilon} \left(\frac{b^2}{c}-\frac bc\right)}}{\sqrt{\frac{2\pi^2}{\varepsilon} \left(\frac{b^2}{c}-1\right)}}\left|\frac{\frac{\pi^2}{\varepsilon} \left(\frac{2b^2}{c}-1-\frac bc\right)+\pi im}{\sqrt{\frac{2\pi^2}{\varepsilon} \left(\frac{b^2}{c}-1\right)}}\right|^{-3}\right)\right).
\end{align*}
Plugging in $\Omega(m,n)$ yields the claim. The main term here simplifies to 
$$ \Lambda_1   \varepsilon^{1-\frac{1}{2}\sum\limits_{j=1}^{N}\alpha_j} e^{ \frac{2\pi^2}{\varepsilon}\Lambda_2 }$$
since $4\pi m\varepsilon\rightarrow 0$ as $n\to \infty$. 
\end{proof}

\section{Bounds away from the dominant pole} \label{Section: away from domminant pole}
In this section we investigate the contribution of $f_m$ away from the dominant pole at $q=1$, and show that it forms part of the error term. Recall that from \eqref{f_m with cancelled residue part} we have
$$ f_m(\tau) = \lim_{r\rightarrow 0^+} -i\left(\sum_{\ell=0}^{s}\int_{\Gamma_{\ell,r}} f(z;\tau) \sin(2\pi mz) dz  \right).$$
One immediately sees that
\begin{align*}
&\left|\sum\limits_{\ell=0}^{s} \int_{\Gamma_{\ell,r}}f(z;\tau) \sin(2\pi mz) dz\right| 
\ll  \sum\limits_{\ell=0}^{s}  \int_{\Gamma_{\ell,r}}\left|f(z;\tau) \sin(2\pi mz)\right| dz .
\end{align*}
Consider 
$$ \left|f(z;\tau) \sin(2\pi mz)\right| = \left|\prod_{j=1}^{N} \eta(a_j \tau)^{\alpha_j}  \frac{\vartheta( z;  \tau)}{\vartheta(b z; c \tau)} \right| \left| \sin(2\pi mz)\right| \ll  \left|\prod_{j=1}^{N} \eta(a_j \tau)^{\alpha_j} \right| \left| \frac{\vartheta( z;  \tau)}{\vartheta(b z; c \tau)} \right|$$
away from the dominant pole.
We begin with the term $\prod_{j=1}^{N} \eta(a_j \tau)^{\alpha_j}$. As in \cite{males2020asymptotic} we write 
\begin{equation*}
	\prod_{j=1}^N \eta(a_j \tau)^{\alpha_j} = \prod_{j=1}^{N_1} \eta(x_j \tau)^{\gamma_j}  \prod_{k=1}^{N_2} q^{-\frac{y_k \delta_k}{24}} P(q^{y_k})^{\delta_k}.
\end{equation*}
Using Lemma \ref{Lemma: transformation of eta} we see that
\begin{equation*}
	\eta\left(\frac{i x_j \varepsilon}{2 \pi}\right)^{\gamma_j} \ll \left(\frac{2 \pi}{x_j \beta}\right)^{\frac{\gamma_j}{2}} e^{-\frac{\pi^2 \gamma_j}{6 x_j \beta}}. 
\end{equation*}
By \eqref{Equation: bound on P(q)} we also obtain that
\begin{equation*}
	|P(q^{y_k})| \ll n^{-\frac{1}{4}} \exp\left[\frac{2\pi}{y_k \beta} \left(\frac{\pi}{12} - \frac{1}{2 \pi}\left(1 - \frac{1}{\sqrt{1+m^{-\frac{2}{3}}}}\right)\right)\right].
\end{equation*}
Therefore we find 
\begin{multline*} 
	\prod_{j=1}^N  \eta\left(\frac{ia_j\varepsilon}{2\pi}\right)^{\alpha_j} \ll  \left(\prod_{j=1}^{N_1} \frac{2 \pi}{x_j \beta}\right)^{\frac{\gamma_j}{2}} e^{-\sum\limits_{j=1}^{N_1}\frac{\pi^2 \gamma_j}{6 x_j \beta}} \\ \times \prod_{k=1}^{N_2} n^{-\frac{\delta_k}{4}} \exp\left[\frac{2 \pi \delta_k}{y_k \beta} \left(\frac{\pi}{12} - \frac{1}{2 \pi}\left(1 - \frac{1}{\sqrt{1+m^{-\frac{2}{3}}}}\right)\right)\right],
\end{multline*}
and thus we obtain
\begin{multline}\label{Eqn: eta bound}
 \left|\prod_{j=1}^{N} \eta(a_j \tau)^{\alpha_j} \right| \, \left|\frac{\vartheta( z;  \tau)}{\vartheta(b z; c \tau)}\right| 
 \ll  \Bigg| \left(\prod_{j=1}^{N_1} \frac{2 \pi}{x_j \beta}\right)^{\frac{\gamma_j}{2}} e^{-\sum\limits_{j=1}^{N_1}\frac{\pi^2 \gamma_j}{6 x_j \beta}}  \prod_{k=1}^{N_2} n^{-\frac{\delta_k}{4}} \\
 \times \exp\left[\frac{2 \pi \delta_k}{y_k \beta} \left(\frac{\pi}{12} - \frac{1}{2 \pi}\left(1 - \left(1+m^{-\frac{2}{3}}\right)^{-\frac{1}{2}}\right)\right)\right]\Bigg|  \left|\frac{\vartheta( z;  \tau)}{\vartheta(b z; c \tau)}  \right|.
\end{multline}
Plugging in \eqref{Equation: summation rep for theta}, using Lemma \ref{Lemma: transformation of theta}, and some rearranging leads to
\begin{align}\label{Eqn: bound}
	 \left|\frac{\vartheta( z;  \tau)}{\vartheta(b z; c \tau)}  \right| =& \left|q^{-\frac{c}{8}}\right| \frac{\left|\vartheta( z;  \tau)\right| }{\left|\sum\limits_{\kappa \in \Z} (-1)^\kappa q^{c\frac{\kappa^2+\kappa}{2}} \zeta^{b\kappa}\right|}\ll  \sqrt{\frac{2\pi}{\beta}} e^{-\frac{2\pi^2}{\beta}\min\limits_{z\in\Gamma_{\ell,r}}\left(z-\frac 12\right)^2} \left|\sum_{\kappa  \in \Z} (-1)^\kappa e^{-\frac{2\pi^2}{\varepsilon}\left(\kappa^2+(1-2z)\kappa\right)}\right| \notag\\
	\ll& \beta^{-\frac 12} e^{-\frac{2\pi^2}{\beta}\min\limits_{z\in\Gamma_{\ell,r}}\left(z-\frac 12\right)^2} 
	\ll n^{\frac 14} e^{-\frac{2\pi^2}{\beta}\min\limits_{z\in\Gamma_{\ell,r}}\left(z-\frac 12\right)^2}\ll n^{\frac 14}.
\end{align}

Defining  
$$ \mathcal{B}(m,n) \coloneqq n^{  \frac 14 - \sum\limits_{k=1}^{N_2}\frac{\delta_k }{4}} \prod_{j=1}^{N_1} \left(\frac{2 \pi}{x_j \beta}\right)^{\frac{\gamma_j}{2}},  $$
equations \eqref{Eqn: eta bound} and \eqref{Eqn: bound} imply that for $r\rightarrow0^+$
\begin{multline*}
\left|\sum\limits_{\ell=0}^{s} \int_{\Gamma_{\ell,r}}f(z;\tau) \sin(2\pi mz) dz\right| \\
\ll \sum_{\ell=0}^{s} \mathcal{B}(m,n) \exp\left[\sum_{k=1}^{N_2}\frac{2 \pi \delta_k}{y_k \beta} \left(\frac{\pi}{12} - \frac{1}{2 \pi}\left(1 - \frac{1}{\sqrt{1+{m^{-\frac{2}{3}}}}}\right)\right)- \sum_{j=1}^{N_1} \frac{\pi^2 \gamma_j}{6 x_j \beta}\right].
\end{multline*}

Hence, away from the dominant pole at $q=1$ we have shown the following proposition.
\begin{proposition} \label{Prop:bound fm away from dominant pole}
	For $1\leq x\leq\frac{\pi m^\frac{1}{3}}{\beta}$ we have that
	\begin{align*}
	\Bigg|f_m\left(\frac{i\varepsilon}{2\pi}\right)\Bigg| \ll& 	(s+1)\mathcal{B}(m,n)  \exp\left[\sum_{k=1}^{N_2}\frac{2 \pi \delta_k}{y_k \beta} \left(\frac{\pi}{12} - \frac{1}{2 \pi}\left(1 - \frac{1}{\sqrt{1+{m^{-\frac{2}{3}}}}}\right)\right)- \sum_{j=1}^{N_1} \frac{\pi^2 \gamma_j}{6 x_j \beta}\right]
	\end{align*}
	as $n\rightarrow\infty$.
\end{proposition}

\section{The Circle Method}\label{Section: Circle method}
In this section we use Wright's variant of the Circle Method to complete the proof of Theorem \ref{Theorem: main}. Cauchy's Theorem implies that
\begin{equation*}
c(m,n) = \frac{1}{2 \pi i} \int_{C} \frac{f_m (\tau)}{q^{n+1}} dq,
\end{equation*}
where $C \coloneqq \{ q \in \C \mid |q| = e^{-\beta}  \}$ is a circle centered at the origin of radius less than $1$, with the path taken in the counter-clockwise direction, traversing the circle exactly once. 
Making a change of variables, changing the direction of the path of integration, and recalling that $\varepsilon = \beta(1+ixm^{-\frac{1}{3}})$ we have
\begin{equation*}
c(m,n) = \frac{\beta}{2\pi m^{\frac 13}} \int_{|x| \leq \frac{\pi m^{\frac{1}{3}}}{\beta}} f_m\left(\frac{i \varepsilon}{2 \pi}\right) e^{\varepsilon n} dx.
\end{equation*}
Splitting this integral into two pieces, we have $c(m,n) = M + E$ where
\begin{equation*}
M \coloneqq \frac{\beta}{2 \pi m^{\frac 13}} \int_{|x| \leq 1} f_m\left( \frac{i \varepsilon}{2 \pi}\right) e^{\varepsilon n} dx,
\end{equation*}
and
\begin{equation*}
E \coloneqq \frac{\beta}{2 \pi m^{\frac{1}{3}}} \int_{1 \leq |x| \leq \frac{\pi m^{\frac{1}{3}} }{\beta}} f_m\left(\frac{i \varepsilon}{2 \pi}\right) e^{\varepsilon n} dx.
\end{equation*}

Next we determine the contributions of each of the integrals $M$ and $E$, and see that $M$ contributes to the main asymptotic term, while $E$ is part of the error term.

\subsection{The major arc} 
Considering the contribution of $M$, we obtain the following proposition. 
\begin{proposition}\label{Prop: maj arc}
	We have that
	\begin{align*}
	M = \frac{-i}{2\pi} \Lambda_1  \beta^{2-\frac 12\sum\limits_{j=1}^N\alpha_j} \sqrt{\Lambda_2 }^{-\frac 12 \sum\limits_{j=1}^N \alpha_j } \frac{e^{2 \pi \sqrt{2\Lambda_2  n}}}{2 \pi \left(2\Lambda_2  n\right)^{\frac 14} } + O \left(\beta^{3-\frac 12\sum\limits_{j=1}^N\alpha_j} \frac{e^{2 \pi \sqrt{2\Lambda_2  n}}}{2 \pi \left(2\Lambda_2  n\right)^{\frac 14} }\right)
	\end{align*}
	as $n\rightarrow\infty$.
\end{proposition}
\begin{proof}
	By definition we have that
		\begin{align}\label{equation: M in two parts}
		M = \frac{\beta}{2 \pi m^{\frac 13}}  \Lambda_1  \int_{|x| \leq 1}  \varepsilon^{1-\frac{1}{2}\sum\limits_{j=1}^{N}\alpha_j} e^{ \frac{2\pi^2}{\varepsilon}\Lambda_2 } e^{\varepsilon n} dx +\frac{\beta}{2 \pi m^{\frac 13}} \int_{|x| \leq 1}O\left( \beta^{2-\frac{1}{2}\sum\limits_{j=1}^{N}\alpha_j}   e^{ \frac{2\pi^2}{\varepsilon} \Lambda_2 }\right) e^{\varepsilon n} dx .
	\end{align}
	Making the change of variables $v = 1 +ixm^{-\frac{1}{3}}$ and then $v \mapsto \sqrt{\Lambda_2 }v$ we obtain that the first term equals
	\begin{align}\label{equation: main term M}
	  \frac{-i}{2\pi} \Lambda_1  \beta^{2-\frac 12\sum\limits_{j=1}^N\alpha_j} \sqrt{\Lambda_2 }^{-\frac 12 \sum\limits_{j=1}^N \alpha_j } P_{1-\frac{1}{2}\sum\limits_{j=1}^N \alpha_j , 12\Lambda_2 },
	\end{align}  
	where $$P_{s,k} \coloneqq \int\limits_{\frac{1-im^{-\frac{1}{3}}}{\sqrt{\Lambda_2 }}}^{\frac{1+im^{-\frac{1}{3}}}{\sqrt{\Lambda_2 }}} v^s e^{\pi \sqrt{\frac{kn}{6}} \left(v+\frac{1}{v}\right)} dv.$$
	One may relate $P_{s,k}$ to $I$-Bessel functions in exactly the same way as \cite[Lemma 4.2]{bringmann2016dyson}, making the adjustment for $\sqrt{\Lambda_2 }$ where necessary, to obtain that 
	\begin{align*}
	P_{s,k} = I_{-s-1}\left(\pi \sqrt{\frac{2kn}{3}}\right) + O\left(\exp\left(\pi\sqrt{\frac{kn}{6}}\left(1+\frac{1}{1+m^{-\frac{2}{3}}}\right)\right)\right).
	\end{align*}
	Using the asymptotic behaviour of the $I$-Bessel function given in Lemma \ref{Lemma: asymptotic of I Bessel} we obtain 
	\begin{align*}
	P_{1-\frac{1}{2}\sum\limits_{j=1}^N \alpha_j , 12\Lambda_2 } =
	\frac{e^{2 \pi \sqrt{2\Lambda_2  n}}}{2 \pi \left(2\Lambda_2  n\right)^{\frac 14} } + O\left( \frac{e^{2\pi \sqrt{2\Lambda_2  n}}}{ \left(8\pi^2 \Lambda_2  n\right)^\frac{3}{4} } \right) + O\left(e^{\pi\sqrt{2\Lambda_2  n}\left(1+\frac{1}{1+m^{-\frac{2}{3}}}\right)}\right), 
	\end{align*}
and therefore \eqref{equation: main term M} becomes 
	\begin{multline*}
		 \frac{-i}{2\pi} \Lambda_1  \beta^{2-\frac 12\sum\limits_{j=1}^N\alpha_j} \sqrt{\Lambda_2 }^{-\frac 12 \sum\limits_{j=1}^N \alpha_j } \frac{e^{2 \pi \sqrt{2\Lambda_2  n}}}{2 \pi \left(2\Lambda_2  n\right)^{\frac 14} } +O\left(\beta^{2-\frac{1}{2}\sum\limits_{j=1}^{N}\alpha_j} \frac{e^{2\pi \sqrt{2\Lambda_2  n}}}{ \left(8\pi^2 \Lambda_2  n\right)^\frac{3}{4} }\right)  \\
		+O\left(\beta^{2-\frac{1}{2}\sum\limits_{j=1}^{N}\alpha_j} e^{\pi\sqrt{2\Lambda_2  n}\left(1+\frac{1}{1+m^{-\frac{2}{3}}}\right)}\right) .
	\end{multline*}
Analogously, the second term of \eqref{equation: M in two parts} is
\begin{align*}
		\frac{-i}{2\pi}  \beta^{3-\frac 12\sum\limits_{j=1}^N\alpha_j} \sqrt{\Lambda_2 }^{-1} P_{0, 12\Lambda_2 } = O \left(\beta^{3-\frac 12\sum\limits_{j=1}^N\alpha_j} \frac{e^{2 \pi \sqrt{2\Lambda_2  n}}}{2 \pi \left(2\Lambda_2  n\right)^{\frac 14} }\right). 
\end{align*}
This yields
\begin{align*}
	M =& \frac{-i}{2\pi} \Lambda_1  \beta^{2-\frac 12\sum\limits_{j=1}^N\alpha_j} \sqrt{\Lambda_2 }^{-\frac 12 \sum\limits_{j=1}^N \alpha_j } \frac{e^{2 \pi \sqrt{2\Lambda_2  n}}}{2 \pi \left(2\Lambda_2  n\right)^{\frac 14} } +O\left(\beta^{2-\frac{1}{2}\sum\limits_{j=1}^{N}\alpha_j} \frac{e^{2\pi \sqrt{2\Lambda_2  n}}}{ \left(8\pi^2 \Lambda_2  n\right)^\frac{3}{4} }\right)  \\
	&+O\left(\beta^{2-\frac{1}{2}\sum\limits_{j=1}^{N}\alpha_j} e^{\pi\sqrt{2\Lambda_2  n}\left(1+\frac{1}{1+m^{-\frac{2}{3}}}\right)}\right) + O \left(\beta^{3-\frac 12\sum\limits_{j=1}^N\alpha_j} \frac{e^{2 \pi \sqrt{2\Lambda_2  n}}}{2 \pi \left(2\Lambda_2  n\right)^{\frac 14} }\right).
\end{align*}
This finishes the proof.
\end{proof}

\subsection{The error arc} Finally, we bound $E$ as follows.
\begin{proposition} \label{Prop: error arc}
	We have $E \ll M$ as $n\rightarrow\infty$.
\end{proposition}
\begin{proof}
	By Proposition \ref{Prop:bound fm away from dominant pole} we have 
	\begin{align*}
		E \ll& \frac{\beta}{2 \pi m^{\frac 13}}  (s+1)\mathcal{B}(m,n)  \exp\left[\sum_{k=1}^{N_2}\frac{2 \pi \delta_k}{y_k \beta} \left(\frac{\pi}{12} - \frac{1}{2 \pi}\left(1 - \frac{1}{\sqrt{1+{m^{-\frac{2}{3}}}}}\right)\right)- \sum_{j=1}^{N_1} \frac{\pi^2 \gamma_j}{6 x_j \beta}\right] \\
		&\times  e^{\beta n} \int_{1 \leq |x| \leq \frac{\pi m^{\frac 13}}{\beta}} e^{\beta n i x m^{-\frac 13}} dx \\
		=& \frac{s+1}{ \pi }  \mathcal{B}(m,n)  \exp\left[\pi\sqrt{2n}\left(1-\sum_{j=1}^{N} \frac{\alpha_j}{12 a_j }\right)-\sum_{k=1}^{N_2}\frac{ \delta_k}{y_k \beta} \left(1 - \frac{1}{\sqrt{1+{m^{-\frac{2}{3}}}}}\right)\right],
	\end{align*}
	where we trivially estimate the final integral. Using $1-\sum_{j=1}^{N} \frac{\alpha_j}{12 a_j }<2\sqrt{\Lambda_2 }$ 
	the result follows immediately by comparing to $M$ and therefore also finishes the proof of Theorem \ref{Theorem: main}. 
\end{proof}

\section{Further questions}
We end by briefly commenting on some related questions that could be the subject of further research.

\begin{enumerate}
	\item Here we only discussed the case of eta-theta quotients with simple poles. A natural question to ask is: does a similar story hold for functions with higher order poles? The situation is of course expected to be more complicated, in particular finding Fourier coefficients with the method presented here seems to be much more difficult. One could attempt to build a framework by following the definitions of Fourier coefficients given in \cite[Section 8]{dabholkar2012quantum}.
	
	For example, in \cite{ManRo} Manschot and Rolon study a Jacobi form with a double pole related to $\chi_y$-genera of Hilbert schemes on K$3$. They obtain bivariate asymptotic behavior in a similar flavor to those here. Can one extend this family?
	
	\item Although the functions considered in the present paper provide a wide family of results, it should be possible to extend the method to other related families of functions. In particular, it would be instructive to consider similar approaches for prototypical examples of mock Jacobi forms.
\end{enumerate}

\end{document}